\documentclass[12pt]{article}
\usepackage{t1enc}
\usepackage[latin1]{inputenc}
\usepackage[english]{babel}
\usepackage{amsmath,amsthm}
\usepackage{amsfonts}
\usepackage{latexsym}
\usepackage[dvips]{graphicx}
\usepackage{graphicx}
\usepackage[natural]{xcolor}
\newtheorem{thm}{Theorem}[section]

\newtheorem{claim}{Claim}
\newtheorem{lem}[thm]{Lemma}

\newtheorem{conj}[thm]{Conjecture}

\newtheorem{con}[thm]{Conclusion}

\numberwithin{equation}{section}

\usepackage[varg]{txfonts}
\usepackage{graphicx, epsfig, subfigure}
\usepackage{enumerate} 
\usepackage[square, numbers, sort&compress]{natbib} 

\numberwithin{equation}{section}

\begin{document}
\title{A conjecture on equitable vertex arboricity of graphs\thanks{This research is supported by the National Natural Science Foundation of China (No. 11101243, 11201440).}}
\author{Xin Zhang\unskip\textsuperscript{a}\thanks{Email address: xzhang@xidian.edu.cn},\;Jian-Liang Wu\unskip\textsuperscript{b}\thanks{Email address: jlwu@sdu.edu.cn}\\[.5em]
{\small \textsuperscript{a}\unskip Department of Mathematics, Xidian University, Xi'an 710071, P. R. China}\\
{\small \textsuperscript{b}\unskip School of Mathematics, Shandong University, Jinan 250100, P. R. China}\\
}
\date{}
\maketitle

\begin{abstract}\baselineskip  0.6cm
Wu, Zhang and Li \cite{WZL} conjectured that the set of vertices of any simple graph $G$ can be equitably partitioned into $\lceil(\Delta(G)+1)/2\rceil$ subsets so that each of them induces a forest of $G$. In this note, we prove this conjecture for graphs $G$ with $\Delta(G)\geq |G|/2$.\\[.5em]
Keywords: Equitable vertex arboricity; Relaxed coloring; Tree coloring; Maximum degree
\end{abstract}

\baselineskip 0.6cm

\section{Introduction}

All graphs considered in this paper are finite, undirected, loopless and without multiple edges. For a graph $G$, we use $V(G)$, $E(G)$, $\delta(G)$ and $\Delta (G)$ to denote the vertex set, the edge set, the minimum degree and the maximum degree of $G$, respectively. By $\alpha'(G)$ and $G^c$, we denote the largest size of a matching in the graph $G$ and the completement graph of $G$. For other basic undefined concepts we refer the reader to
\cite{BM}.

The \emph{vertex-arboricity} $a(G)$ of a graph $G$ is the minimum number of subsets into which the vertex set $V(G)$ can be partitioned so that each subset induces a forest. This notation was first introduced by Chartrand, Kronk and Wall \cite{CKW} in 1968, who named it point-arboricity and proved that $a(G)\leq \lceil(\Delta(G)+1)/2\rceil$ for every graph $G$. Recently, Wu, Zhang and Li \cite{WZL} introduced the equitable version of vertex arboricity.
If the set of vertices of a graph $G$ can be equitably partitioned into $k$ subsets (i.e.\,\,the size of each subset is either $\lceil|G|/k\rceil$ or $\lfloor|G|/k\rfloor$) such that each subset of vertices induce a forest of $G$, then we call that $G$ admits an \emph{equitable $k$-tree-coloring}. The minimum integer $k$ such that $G$ has an equitable $k$-tree-coloring is the \emph{equitable vertex arboricity} $a_{eq}(G)$ of $G$. As an extension of the result of Chartrand, Kronk and Wall on vertex arboricity, Wu, Zhang and Li {\rm \cite{WZL}} raised the following conjecture and they proved it for complete bipartite graphs, graphs with maximum average degree less than 3, and graphs with maximum average degree less than $10/3$ and maximum degree at least 4.
\begin{conj}\label{wzlconj}
 $a_{eq}(G)\leq \lceil\frac{\Delta(G)+1}{2}\rceil$ for every simple graph $G$.
\end{conj}
In this note, we establish this conjecture for graphs $G$ with $\Delta(G)\geq |G|/2$.

\section{Main results and the proofs}

For convenience, we set $\Gamma(G)=\lceil\frac{\Delta(G)+1}{2}\rceil$ throughout this section. To begin with, we introduce two useful lemmas of Chen, Lih and Wu.
\begin{lem}{\rm \cite{CLW}}\label{dcon}
If $G$ is a disconnected graph, then $\alpha'(G)\geq \delta(G)$.
\end{lem}

\begin{lem}{\rm \cite{CLW}}\label{con}
If $G$ is a connected graph such that $|G|>2\delta(G)$, then $\alpha'(G)\geq \delta(G)$.
\end{lem}

\begin{lem}\label{cycle}
If $G$ is a connected graph with $\delta(G)\geq 2$, then $G$ contains a cycle of length at least $\delta(G)+1$.
\end{lem}

\begin{proof}
Consider the longest path $P=[v_0v_1\ldots v_k]$ in $G$. We see immediately that $N(v_0)\subseteq V(P)$, because otherwise we would construct a longer path. Let $v_i$ be a neighbor of $v_0$ so that $i$ is maximum. Since $\delta(G)\geq 2$, $C=[v_0v_1\ldots v_iv_0]$ is a cycle of length $i+1\geq \delta(G)+1$.
\end{proof}

In what follows, we prove three independent theorems, which together imply Conjecture \ref{wzlconj} for graphs $G$ with $\Delta(G)\geq |G|/2$.

\begin{thm}
If $\Delta(G)\geq \frac{2}{3}|G|-1$, then $a_{eq}(G)\leq \Gamma(G)$.
\end{thm}

\begin{proof}
If $\Delta(G)=|G|-1$, then $a_{eq}(G)\leq \Gamma(G)$ and this upper bound can be attained by the complete graphs, since we can arbitrarily partition $V(G)$ into $\Gamma(G)$ subsets so that each of them consists of one or two vertices, thus we assume $\Delta(G)\leq |G|-2$.
Since $\Delta(G)+\delta(G^c)=|G|-1$ and $\Delta(G)\geq \frac{2}{3}|G|-1$, $|G^c|\geq 3\delta(G^c)$ and $\delta(G^c)\geq |G^c|-2\Gamma(G)$.
By Lemmas \ref{dcon} and \ref{con}, we have $\alpha'(G^c)\geq \delta(G^c)$, so there exists a matching $M=[x_1y_1,\ldots,x_{\delta}y_{\delta}]$ of size $\delta:=\delta(G^c)$ in $G^c$. Since $|G^c|\geq 3\delta(G^c)$, $|V(G^c)\setminus V(M)|\geq \delta$, thus we can select $\delta$ distinct vertices $z_1,\ldots,z_{\delta}$ among $V(G^c)\setminus V(M)$. Denote $\beta=|G^c|-2\Gamma(G)$ and $\mu=3\Gamma(G)-|G^c|$. Since $|G|-2\geq \Delta(G)\geq \frac{2}{3}|G|-1$, $\beta,\mu\geq 0$. We now use $\beta$ colors to color $3\beta$ vertices of $G$ so that the $i$-th color class consists of the three vertices $x_i,y_i$ and $z_i$, and then use $\mu$ colors to color the remaining $2\mu$ vertices of $G$ so that each color class consists of two vertices. One can check that each color class of $G$ induces a (linear) forest and the coloring of $G$ is equitable. Therefore, $a_{eq}(G)\leq \beta+\mu=\Gamma(G)$.
\end{proof}

\begin{thm}
If $\frac{2}{3}|G|-1>\Delta(G)\geq \frac{2}{3}|G|-2$, then $a_{eq}(G)\leq \Gamma(G)$.
\end{thm}

\begin{proof}
If $|G|\leq 3$, then the result is trivial, so we assume $|G|\geq 4$. If $|G|=3k$, then $\Delta(G)=2k-2$ and $\delta(G^c)=k+1$, since $\frac{2}{3}|G|-1>\Delta(G)\geq \frac{2}{3}|G|-2$ and $\Delta(G)+\delta(G^c)=|G|-1$. By Lemmas \ref{dcon} and \ref{con}, we have $\alpha'(G^c)\geq \delta(G^c)>k$. Let $M_1=[x_{11}y_{11},\ldots,x_{1k}y_{1k}]$ be a matching of $G^c$. We now partition the vertices of $G$ into $k$ subsets so that the $i$-th
subset consists of the vertices $x_{1i},y_{1i}$ and one another vertex different from the vertices in $V(M_1)$. It is easy to check that this is an equitable partition so that each subset induces a (linear) forest, therefore, $a_{eq}(G)\leq k=\Gamma(G)$. If $|G|=3k+2$, then $\Delta(G)=2k$ and $\delta(G^c)=k+1$. This also implies, by Lemma \ref{dcon} and \ref{con}, that $\alpha'(G^c)\geq \delta(G^c)>k$. Let $M_2=[x_{21}y_{21},\ldots,x_{2k}y_{2k}]$ be a matching of $G^c$. We now partition the vertices of $G$ into $k+1$ subsets so that the $i$-th
subset with $i\leq k$ consists of the vertices $x_{2i},y_{2i}$ and one another vertex different from the vertices in $V(M_2)$ and the $(k+1)$-th subset consists of two vertices in $V(G)\setminus V(M_2)$. It is easy to check that this is an equitable partition so that each subset induces a (linear) forest, therefore, $a_{eq}(G)\leq k+1=\Gamma(G)$. If $|G|=3k+1$, then $\Delta(G)=2k-1$ and $\delta(G^c)=k+1$. By Lemmas \ref{dcon} and \ref{con}, we have $\alpha'(G^c)\geq \delta(G^c)$. Let $M_3=[x_{31}y_{31},\ldots,x_{3(k+1)}y_{3(k+1)}]$ be a matching of $G^c$. If $x_{31}$ has a neighbor in $G^c$ among $\{x_{32},y_{32},\ldots,x_{3(k+1)},y_{3(k+1)}\}$ (without loss of generality, assume that $x_{31}x_{32}\in E(G^c)$), then
we can partition the vertices of $G$ into $k$ subsets so that the the first subset consists of the four vertices $x_{31},y_{31},x_{32}$ and $y_{32}$ and
the $i$-th subset with $2\leq i\leq k$ consists of the vertices $x_{3(i+1)},y_{3(i+1)}$ and one another vertex different from the vertices in $V(M_2)$. One can check that this is an equitable partition so that each subset induces a (linear) forest, therefore, $a_{eq}(G)\leq k=\Gamma(G)$. Hence, we shall assume that $x_{31}x_{3j},x_{31}y_{3j}\not\in E(G^c)$ for each $2\leq j\leq k+1$. Since $d_{G^c}(x_{31})\geq \delta(G^c)=k+1$ and $|G^c|=3k+1$, $x_{31}z\in E(G^c)$ for each $z\in V(G^c)\setminus V(M_3)$. Similarly, we shall assume that $y_{31}z\in E(G^c)$ for each $z\in V(G^c)\setminus V(M_3)$, because otherwise we return to a case we have considered before. We now partition the vertices of $G$ into $k$ subsets so that the the first subset consists of the two vertices $x_{31},y_{31}$ and two distinct vertices $z_1,z_2\in V(G^c)\setminus V(M_3)$ and
the $i$-th subset with $2\leq i\leq k$ consists of the vertices $x_{3i},y_{3i}$ and one another vertex different from the vertices in $V(M_2)$. One can again check that this is an equitable partition so that each subset induces a (linear) forest, therefore, $a_{eq}(G)\leq k=\Gamma(G)$.
\end{proof}

\begin{thm}
If $\frac{2}{3}|G|-2>\Delta(G)\geq \frac{1}{2}|G|$, then $a_{eq}(G)\leq \Gamma(G)$.
\end{thm}

\begin{proof}
Since $\Delta(G)+\delta(G^c)=|G|-1$ and $\Delta(G)\geq \frac{1}{2}|G|$, $|G^c|\geq 2\delta(G^c)+2$. We split our proof into two cases.\\[.5em]
\noindent\emph{Case 1: $G^c$ is connected.}\\[.5em]
\indent Since $|G^c|\geq 2\delta(G^c)+2>2\delta(G^c)$, there exists a path $P=[x_0,x_1,\ldots,x_{2\delta}]$ of length $2\delta:=2\delta(G^c)$ in $G^c$ (see [1, Exercise 4.2.9]). Denote $\beta=|G|-3\Gamma(G)$ and $\mu=4\Gamma(G)-|G|$. Since $\frac{2}{3}|G|-2>\Delta(G)\geq \frac{1}{2}|G|$, $\beta,\mu\geq 1$.
Since $2\Gamma(G)>\Delta(G)=|G|-\delta(G^c)-1$, $\delta(G^c)\geq |G|-2\Gamma(G)=2\beta+\mu$. Thus, the vertex sets
$V_i=\{x_{4i-4},x_{4i-3},x_{4i-2},x_{4i-1}\}$ with $1\leq i\leq \beta$ and $U_i=\{x_{4\beta+2i-2},x_{4\beta+2i-1}\}$ with $1\leq i\leq \mu$ are well defined. Note that
$V(P)\supseteq \bigcup_{i=1}^{\beta}V_i \cup \bigcup_{i=1}^{\mu}U_i$. Since $|G|-4\beta-3\mu=\mu$, $|G-\bigcup_{i=1}^{\beta}V_i - \bigcup_{i=1}^{\mu}U_i|=\mu$.
Let $V(G)\setminus \left(\bigcup_{i=1}^{\beta}V_i \cup \bigcup_{i=1}^{\mu}U_i\right)=\{y_1,\ldots,y_{\mu}\}$ and let $W_i=U_i \cup \{y_i\}$ with $1\leq i\leq \mu$. We now partition the vertices of $G$ into $\beta+\mu$ subsets $V_1,\ldots,V_{\beta},W_1,\ldots,W_{\mu}$. One can check that this is an equitable partition so that each subset induces a (linear) forest, therefore, $a_{eq}(G)\leq \beta+\mu=\Gamma(G)$.\\[.5em]
\noindent\emph{Case 2: $G^c$ is disconnected.}\\[.5em]
\indent Let $G_1,\ldots,G_t$ be the components of $G^c$ with $t\geq 2$. Since $\Delta(G)+\delta(G^c)=|G|-1$ and $\Delta(G)<\frac{2}{3}|G|-2$, $\min\{\delta(G_1),\dots.\delta(G_t)\}\geq \delta(G^c)\geq 2$. This implies, by Lemma \ref{cycle}, that $G_i$ contains a cycle $C_i=[x^i_0 x^i_1 \ldots x^i_{l(C_i)} x^i_0]$ of length $l(C_i)+1\geq \delta(G_i)+1$ for each $1\leq i\leq t$. Let $V^i_j=\{x^i_{4j-4},x^i_{4j-3},x^i_{4j-2},x^i_{4j-1}\}$
with $1\leq i\leq t$ and $1\leq j\leq n_i$, in which $4n_i-1\leq l(C_i)$ and $n_1+\ldots+n_t=\beta$. Note that $V^i_j$ is well defined by Claim \ref{cm}.

\begin{claim}\label{cm}
$\Sigma_{i=1}^t \lfloor\frac{\delta(G_i)+1-4n_i}{2}\rfloor \geq \mu$.
\end{claim}

\begin{proof}
Otherwise, $\delta(G^c)\leq \frac{1}{2}t\delta(G^c)\leq \frac{1}{2}\Sigma_{i=1}^t \delta(G_i)\leq\Sigma_{i=1}^t \lfloor\frac{\delta(G_i)+1}{2}\rfloor< 2\beta+\mu=|G|-2\Gamma(G)\leq |G|-\Delta(G)-1$, contradicting to $\Delta(G)+\delta(G^c)=|G|-1$.
\end{proof}

We conclude, by Claim \ref{cm}, that there exists a matching $M$ of size at least $\mu$ in $G^c-\bigcup_{i=1}^t \bigcup_{j=1}^{n_i} V^i_j$.
Therefore, we can partition the vertices of $G$ into $\beta+\mu$ subsets so that the $i$-th subset with $1\leq i\leq \mu$ consists of a pair of vertices matched under $M$ and one vertex in $V(G)\setminus \left(V(M)\cup \bigcup_{i=1}^t \bigcup_{j=1}^{n_i} V^i_j \right)$ and the last $\beta$ subsets are $V^1_1,\ldots,V^1_{n_1},\ldots, V^t_{1},\ldots, V^t_{n_t}$. One can check that this is an equitable partition so that each subset induces a (linear) forest, therefore, $a_{eq}(G)\leq \beta+\mu=\Gamma(G)$.
\end{proof}

From the proofs of the above three theorems, we can immediately deduce the following conclusions.

\begin{con}
If $G$ is a simple graph with $\Delta(G)\geq \frac{1}{2}|G|$, then $|V(G)|$ can be equitably partitioned into $\Gamma(G)$ subsets so that each of them induces a linear forest of $G$, i.e., the equitable linear vertex arboricity of $G$ is at most $\Gamma(G)$, and the upper bound $\Gamma(G)$ is sharp.
\end{con}

\begin{con}
An equitable $\Gamma(G)$-tree-coloring of any simple graph $G$ can be constructed in linear time.
\end{con}

\end{document}